\documentclass[11pt, letterpaper,leqno]{amsart}
\usepackage{amssymb}
\usepackage{mathtools}
\usepackage{amsthm}
\usepackage{graphicx, tikz}
\usepackage{enumerate}
\usepackage{enumitem}
\usepackage{xcolor}
\usepackage{hyperref}

\numberwithin{equation}{section}

\newtheorem{theo}{Theorem}[section]

\newtheorem{lm}[theo]{Lemma}

\DeclareMathOperator{\D}{\mathbb{D}}

\DeclareMathOperator{\C}{\mathbb{C}}
\DeclareMathOperator{\R}{\mathbb{R}}
\DeclareMathOperator{\Ha}{\mathbb{H}}

\theoremstyle{definition}

\newtheorem{rmk}[theo]{Remark}

\title[Symmetrization and semigroups of holomorphic functions]{Symmetrization and the rate of convergence of semigroups of holomorphic functions}
\author[Dimitrios Betsakos]{Dimitrios Betsakos}
\author[Argyrios Christodoulou]{Argyrios Christodoulou}
\address{Department of Mathematics, Aristotle University of Thessaloniki, 54124, Thessaloniki, Greece}
\email{betsakos@math.auth.gr}
\address{Department of Mathematics, Aristotle University of Thessaloniki, 54124, Thessaloniki, Greece}
\email{argyriac@math.auth.gr}


\date{}

\begin{document}

	\begin{abstract}
Let $(\phi_t)$, $t\ge 0$,  be a semigroup of holomorphic self-maps of the unit disk $\D$. Let $\Omega$ be its Koenigs domain and $\tau\in \partial \D$ be its Denjoy-Wolff point. Suppose that $0\in \Omega$ and let $\Omega^\sharp$ be the Steiner symmetrization of $\Omega$ with respect to the real axis. Consider the semigroup $(\phi_t^\sharp)$ with Koenigs domain $\Omega^\sharp$ and let $\tau^\sharp$ be its Denjoy-Wolff point. We show that, up to a multiplicative constant, the rate of convergence of $(\phi_t^\sharp)$ is slower than that of $(\phi_t)$; that is, for every $t>0$, $|\phi_t(0)-\tau|\leq 4\pi\,
|\phi_t^\sharp(0)-\tau^\sharp|$. The main tool for the proof is the harmonic measure.	
\end{abstract}

	\maketitle


\section{{\bf Introduction}}
We will prove that if we apply Steiner symmetrization to the Koenigs domain of a semigroup of holomorphic functions, then we obtain a new semigroup with slower rate of convergence towards the Denjoy-Wolff point. We start this introduction with the necessary background on semigroups and symmetrization.

\medskip

A continuous, one-parameter semigroup of holomorphic functions in the unit disk $\D$ 
(from now on, a \textit{semigroup in} $\D$) is a family of holomorphic functions $\phi_t:\D\to\D$, $t\ge 0$,  with the following properties:
\begin{enumerate}
	\item[(i)] $\phi_0(z)=z$, for all $z\in \D$,
	\item[(ii)] $\phi_{t+s}(z)=\phi_t\circ\phi_s(z)$, for all $t,s\ge0$ and all $z\in \D$,
	\item[(iii)] $t\mapsto \phi_t$ is continuous in the topology of uniform convergence of compacta in $\D$.  
\end{enumerate}
The rich theory of semigroups, including remarkable recent developments, is presented
in the books \cite{Aba}, \cite{book}, \cite{ES}. We review some basic facts that we will need later.

By the continuous version of the classical Denjoy-Wolff theorem \cite[Theorem 5.5.1]{Aba}, if the semigroup is not a group of elliptic conformal automorphisms of $\D$, there exists a unique point $\tau\in \overline{\D}$ (called {\it the Denjoy-Wolff point} of the semigroup) such that
	\begin{equation}
	\lim\limits_{t\to+\infty}\phi_t(z)=\tau, \quad\text{for all }z\in\mathbb{D}.
	\end{equation}
The semigroup is called {\it non-elliptic} if $\tau\in \partial \D$; in the present work, we will deal only with non-elliptic semigroups. 

\medskip

If $(\phi_t)$ is a non-elliptic semigroup, then \cite[Theorem 5.7.2]{Aba} there exists a unique conformal mapping $h:\D\to\C$, the {\it Koenigs function} of the semigroup, such that $h(0)=0$ and 
\begin{equation}
h(\phi_t(z))=h(z)+t,  \quad\text{for all }z\in\mathbb{D}\text{ and for all }t\ge 0.
\end{equation} 
The simply connected domain $\Omega\vcentcolon=h(\D)$ is the {\it Koenigs domain} of the semigroup. It is {\it convex in the positive direction}; namely, for every $t\geq 0$,
$\Omega+t\subset \Omega$.

\medskip

Conversely, suppose that $\Omega\subsetneq \C$ is a domain convex in the positive direction and $0\in \Omega$. The half-line $[0,+\infty)$ lies in $\Omega$ and determines a prime end $P_\infty$ of $\Omega$. If $\tau\in \partial \D$, we consider the Riemann map $h:\D\to\Omega$ with $h(0)=0$ and $h(\tau)=P_\infty$ (in the sense of Carath\'eodory's boundary correspondence). We define
\begin{equation}
\phi_t(z)=h^{-1}(h(z)+t):\D\to\D,  \quad\text{for all }z\in\mathbb{D}\text{ and for all }t\ge 0.
\end{equation}

Then $(\phi_t)_{t\geq 0}$ is a non-elliptic semigroup in $\D$ with Koenigs domain $\Omega$ and Denjoy-Wolff point $\tau$. 

\medskip

Thus there is an essentially one-to-one correspondence between non-elliptic semigroups and domains $\Omega$ convex in the positive direction. The geometry of $\Omega$ encodes all the properties of $(\phi_t)$. The theory of semigroups is, to a large extent, the interplay of geometric properties of $\Omega$ and dynamic properties of $(\phi_t)$. One such property of $(\phi_t)$ is its rate of convergence towards the Denjoy-Wolff point $\tau$; that is, the estimation of the quantity $|\phi_t(z)-\tau|$ (for large $t>0$) in terms of geometric properties of $\Omega$ and the position of $z$ in $\D$. This property has been extensively studied for many years; see  \cite{ES},  \cite{book}, \cite{BCD}, and references therein. 

\medskip

Suppose that $(\phi_t), (\tilde{\phi}_t)$ are two non-elliptic semigroups in $\D$ with Koenigs domains $\Omega,\tilde{\Omega}$ and Denjoy-Wolff points $\tau, \tilde{\tau}$, respectively. Bracci \cite{Bra} posed the following problem: How does the inclusion  $\Omega\subset\tilde{\Omega}$ affect the rates of convergence of the semigroups $(\phi_t), (\tilde{\phi}_t)$?
It was proved in \cite{Bra}, \cite{BCK}, \cite{BK} that there is a constant $K>0$ such that for every $t\geq 0$, 
\begin{equation}\label{i4}
|\phi_t(0)-\tau|\leq K\, |\tilde{\phi}_t(0)-\tilde{\tau}|.
\end{equation}
A similar inequality holds if we replace the point $0$ by another point $z\in \D$; in this case, the constant $K$ will depend on $z$.  
So,  the semigroup with larger Koenigs domain has slower rate of convergence.

\medskip

It is a general heuristic principle that a domain monotonicity property (like the one above) indicates that a similar symmetrization property is true, too. So, our purpose is to present and prove such a symmetrization result. If $\Omega\subsetneq \C$ is a domain, we will denote by $\Omega^\sharp$ the Steiner symmetrization of $\Omega$ (with respect to the real axis). It is defined as follows: For any real $x$, the vertical line from $x$ intersects $\Omega$ in a set of disjoint, open, vertical segments of total length
$\ell_\Omega(x)$; then
\begin{equation}
\Omega^\sharp\vcentcolon=\left \{x+iy\in\C: \;|y|< \frac{\ell_\Omega(x)}{2}\right \}.
\end{equation}
Note that $\Omega^\sharp$ is symmetric with respect to the real axis and its intersection with any vertical line is either empty or a symmetric, open, vertical segment. Similarly, for a closed set $F\subset \C$, we define its Steiner symmetrization
\begin{equation}
F^\sharp\vcentcolon=\left \{x+iy\in\C: \;|y|\le\frac{\ell_F(x)}{2}\right \},
\end{equation}
where $\ell_F(x)$ is the length of the intersection of $F$ with the vertical line from $x\in\R$. See Figure \ref{symmetrization figure} for an example Steiner symmetrization. We refer to the books \cite{Hay94}, \cite{Dub} for properties and applications of Steiner symmetrization in geometric function theory. 
If, in addition, $\Omega$ is convex in the positive direction, then (see Lemma \ref{T1L1})
$\Omega^\sharp$ is also convex in the positive direction. It follows that we can consider the semigroups  $(\phi_t), (\phi^\sharp_t)$ having Koenigs domains
 $\Omega, \Omega^\sharp$, respectively. Based on the heuristic principle mentioned above, we state the following result. 

\medskip

\begin{theo}\label{T1}
Let $(\phi_t)$ be a non-elliptic semigroup in $\D$ with Denjoy-Wolff point $\tau$ and Koenigs domain $\Omega$. Let $\Omega^\sharp$ be the Steiner symmetrization of $\Omega$.  Consider the semigroup $(\phi^\sharp_t)$  in $\D$ having 
Denjoy-Wolff point $\tau^\sharp$ and Koenigs domain $\Omega^\sharp$. Then, for every $t\ge 0$,
\begin{equation}
|\phi_t(0)-\tau|\leq 4\pi\,|\phi^\sharp_t(0)-\tau^\sharp|.
\end{equation}
\end{theo}

\medskip

\begin{rmk}\label{R1}
	One can combine Theorem \ref{T1} with the domain monotonicity inequality (\ref{i4})
	to obtain a variety of results for the rate of convergence of semigroups by imposing geometric conditions on the Koenigs domain. For instance,  suppose that $\Omega$ lies in the half-plane $\{z:{\rm Re} z>-1\}$ and that for every $x>-1$, we have $\ell_\Omega(x)<c(x+1)$, where $c$ is a positive constant. Then $\Omega^\sharp$ is contained in the angular domain
	$$
	\Omega^\circ\vcentcolon=\left\{x+iy:x>-1,\;|y|<\frac{c(x+1)}{2}\right\}.		
	$$
	Consider the semigroup $(\phi^\circ_t)$ corresponding to $\Omega^\circ$, having Denjoy-Wolff point $\tau^\circ$ and Keonigs function $h^\circ$.
	By (\ref{i4}) and Theorem \ref{T1},
	\begin{equation}
	|\phi_t(0)-\tau|\leq K\; |\phi^\circ_t(0)-\tau^\circ|,\;\;\;\;\;t\ge 0.
	\end{equation}
	The rate of convergence for  $(\phi^\circ_t)$ has been studied e.g. in \cite{BCD}. 
\end{rmk}

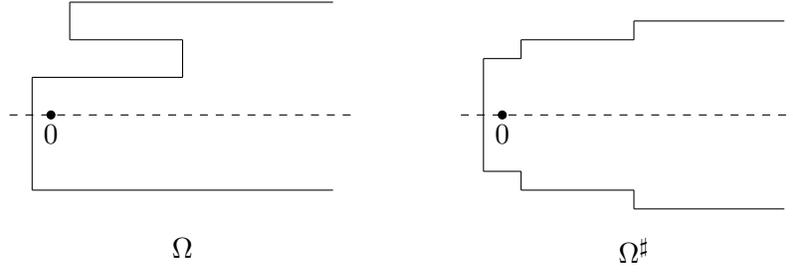
\begin{figure}
\centering
\begin{tikzpicture}

\begin{scope}
\draw[dashed] (-0.3,0) -- (4.3,0);

\draw (0,0.5) -- (0,-1);
\draw (0,-1) -- (4,-1);

\draw (0,0.5) -- (2,0.5);
\draw (2,0.5) -- (2,1);
\draw (2,1) -- (0.5,1);
\draw (0.5,1) -- (0.5,1.5);
\draw (0.5,1.5) -- (4,1.5);

\filldraw[black] (.25,0) circle (1.5pt) node[below] {0};
\node[below] at (2,-1.5) {$\Omega$};
\end{scope}

\begin{scope}[xshift=6cm]
\draw[dashed] (-0.3,0) -- (4.3,0);

\draw (0, 0.75) -- (0, -0.75);
\draw (0, 0.75) -- (0.5, .75);
\draw (0, -0.75) -- (0.5, -.75);

\draw (0.5,0.75) -- (0.5,1);
\draw (0.5,-0.75) -- (0.5,-1);
\draw (0.5,1) -- (2,1);
\draw (0.5,-1) -- (2,-1);

\draw (2,1) -- (2,1.25);
\draw (2,-1) -- (2,-1.25);
\draw (2,1.25) -- (4,1.25);
\draw (2,-1.25) -- (4,-1.25);

\filldraw[black] (.25,0) circle (1.5pt) node[below] {0};
\node[below] at (2,-1.5) {$\Omega^\sharp$};
\end{scope}
\end{tikzpicture}
\label{symmetrization figure}
\caption{A domain $\Omega$ convex in the positive direction and its symmetrization $\Omega^\sharp$.}
\end{figure}

Our second result is a similar theorem for polarization. This is a  geometric transformation defined as follows. 
Let $\Ha$ be the open upper half-plane, and $\mathbb{F}=\Ha^c$ the closed lower half-plane. For a set $A\subset \C$, we define 
\[
A^*=\{\overline{z}\colon z\in A\}, \quad A^-=A\cap \mathbb{F},\quad A^+=A\cap \overline{\Ha}.
\]

The \emph{polarization} of $A$ (with respect to the real axis) is defined as
\[
\widehat{A}\vcentcolon=\left(A\cup A^*\right)^+\cup \left(A\cap A^*\right)^-.
\]

\medskip

See Figure \ref{polarization figure} for an example of domain polarization. We will see (Lemma \ref{T2L2}) that if $D$ is a domain that is convex in the positive direction, then its polarization $\widehat{D}$ is also a domain, convex in the positive direction. So, starting from a non-elliptic semigroup, we can consider its polarized version and prove the following theorem.

\begin{theo}\label{T2}
	Let $(\phi_t)$ be a non-elliptic semigroup in $\D$ with Denjoy-Wolff point $\tau$ and Koenigs domain $\Omega$. Let $\widehat{\Omega}$ be the polarization of $\Omega$.  Consider the semigroup $(\widehat{\phi}_t)$ in $\D$ having 
	Denjoy-Wolff point $\widehat{\tau}$ and Koenigs domain $\widehat{\Omega}$. Then, for every $t\ge 0$,
	\begin{equation}
	|\phi_t(0)-\tau|\leq 2\pi\,|\widehat{\phi}_t(0)-\widehat{\tau}|.
	\end{equation} 
\end{theo}

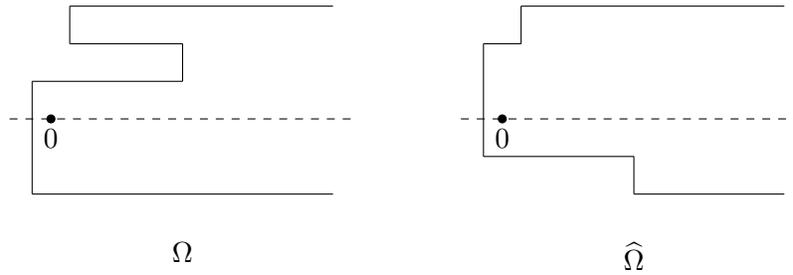
\begin{figure}[h]
\centering
\begin{tikzpicture}
\begin{scope}
\draw[dashed] (-0.3,0) -- (4.3,0);

\draw (0,0.5) -- (0,-1);
\draw (0,-1) -- (4,-1);

\draw (0,0.5) -- (2,0.5);
\draw (2,0.5) -- (2,1);
\draw (2,1) -- (0.5,1);
\draw (0.5,1) -- (0.5,1.5);
\draw (0.5,1.5) -- (4,1.5);

\filldraw[black] (.25,0) circle (1.5pt) node[below] {0};
\node[below] at (2,-1.5) {$\Omega$};
\end{scope}

\begin{scope}[xshift = 6cm]
\draw[dashed] (-0.3,0) -- (4.3,0);

\draw (0,1) -- (0,-.5);
\draw (0,1) -- (0.5,1);
\draw (0,-.5) -- (2,-.5);
\draw (2,-0.5) -- (2,-1);
\draw (.5,1) -- (.5,1.5);
\draw (0.5,1.5) -- (4,1.5);
\draw (2,-1) -- (4,-1);

\filldraw[black] (.25,0) circle (1.5pt) node[below] {0};
\node[below] at (2,-1.5) {$\widehat{\Omega}$};
\end{scope}

\end{tikzpicture}
\label{polarization figure}
\caption{A domain $\Omega$ convex in the positive direction and its polarization $\widehat{\Omega}$.}
\end{figure}

The proofs of the theorems are based on estimates for harmonic measure. 	
A short review of the basic facts about harmonic measure is presented in Section 2.
The proof of Theorems \ref{T1} and \ref{T2} are  in Sections 3 and 4, respectively.


\section{{\bf Preliminaries for Harmonic Measure}}
 Let $\Omega\subset\mathbb{C}$ be a domain with non-polar boundary and let $E$ be a Borel subset of $\partial\Omega$. The harmonic measure of $E$ with respect to $\mathbb{D}$ is the solution of the generalized Dirichlet problem for the Laplacian in $\Omega$ with boundary function $\chi_E$; see \cite[Theorem 4.2.6]{Ran}. For $z\in\Omega$, we are going to use the notation $\omega(z,E,\Omega)$ for the harmonic measure of $E$ with respect to $\Omega$, evaluated at $z\in\Omega$.  
 
 \medskip
 
There are several properties of harmonic measure that will be of use. First of all, the harmonic measure is a conformal invariant. Moreover, it satisfies the following monotonicity property: if $\Omega_1\subset\Omega_2$ are two domains with non-polar boundaries, $E\subset\partial\Omega_1\cap\partial\Omega_2$ is Borel and $z\in\Omega_1$, then $\omega(z,E,\Omega_1)\le\omega(z,E,\Omega_2)$. In addition, for fixed $z\in \Omega$, the harmonic measure  $\omega(z,\cdot,\Omega)$ is a Borel probability measure on $\partial\Omega$.

\medskip

By the Poisson formula, if $E$ is a Borel set on the unit circle, then $\omega(0,E,\D)=\frac{\ell(E)}{2\pi}$, where $\ell$ is the length measure. Using this fact and  conformal invariance, we may compute the harmonic measure in some symmetric situations. For example, for $0<r<1$,
\begin{equation}\label{slit}
\omega(0,[r,1],\D\setminus [r,1])=\frac{2}{\pi}\tan^{-1}\frac{1-r}{2\sqrt{r}}.
\end{equation} 

\medskip

The behaviour of harmonic measure under symmetrization was studied by A. Baernstein.
The main theorem \cite[Theorem 7]{Ba74} involves circular symmetrization, but, as Baernstein remarks in \cite{Ba02}, his star-function method can be adapted to the Steiner symmetrization setting and yields the following theorem; see also \cite{Sol2}, \cite{Ba19}.

\begin{theo}\label{Bae}
Let $\Omega$ be a domain with $0\in \Omega\subset \{z\in\C: {\rm Re}\,z<t\}$, $t\in\R$. Set $L=\{z\in \partial \Omega: {\rm Re}\, z=t\}$. Then
$$
\omega(0,L,\Omega)\leq 	\omega(0,L^\sharp,\Omega^\sharp).
$$
\end{theo}

The following diameter estimate for harmonic measure  was proved in \cite{FRW} and
\cite{Sol2}.

\begin{theo}\label{diam} 
Let $E$ be a Jordan arc in the closure of the unit disk. Let  $d$ be the diameter of $E$.  Let $E_d$ be an arc on the unit circle which
has the same diameter as $E$ (if $d=2$, we take $E_d$ to be a half-circle). Then
\begin{equation}
\omega(0,E,\D\setminus
E)\geq\omega(0,E_d,\D)=\frac{1}{\pi}\sin^{-1}\,\frac{d}{2}.
\end{equation}
\end{theo}

The behaviour of harmonic measure under polarization was studied by A. Yu. Solynin \cite{Sol}. We state here a simplified version of \cite[Theorem 2]{Sol}; we will need it in the proof of Theorem \ref{T2}. 

\begin{theo}\label{pol} 
Let $\Omega\subset\mathbb{C}$ be a domain containing the half-line $[0,\infty)$. Let $A\subset (0,\infty)$ be a non-polar, closed set. Then
$$
\omega(0,A,\Omega\setminus A)\leq 	\omega(0,A,\widehat{\Omega}\setminus A).
$$
\end{theo}


\section{{\bf Proof of Theorem \ref{T1}}}\label{ProofT1}

 The definition of the semigroup $(\phi_t^\sharp)$ is based on the fact that the Steiner symmetrization $ \Omega^\sharp$ of $\Omega$ is
convex in the positive direction. We now  prove this fact. We denote by $L_x$ the vertical line through $x\in\R$ and by $\ell$ the length measure. 

\begin{lm}\label{T1L1}
If $D$ is a domain, convex in the positive direction, then the same is true for $D^\sharp$.	
\end{lm}
\begin{proof}
Let $z_o=x_o+iy_o\in D^\sharp$. We need to show that the horizontal half-line $\{z\in\mathbb{C}\colon \mathrm{Re}\, z >x_o, \; \mathrm{Im}\, z = y_o\}$ lies in $D^\sharp$. Since the domain $D^\sharp$ is symmetric with respect to the real axis, it suffices to assume that $y_o>0$. Because $D^\sharp$ is open, there exists an $\epsilon>0$ so that the closed discs $\overline{D}(z_o,\tfrac{\epsilon}{2})$ and $\overline{D}(\overline{z_o},\tfrac{\epsilon}{2})$ are both contained in $D^\sharp$, which means that
\[
0<2y_o+\epsilon=\left\lvert z_0+i\tfrac{\epsilon}{2} - \left(\overline{z_o}-i\tfrac{\epsilon}{2}\right)\right\rvert <\ell\left(L_{x_o}\cap D\right).
\]
So, for every $c>0$ with $2y_o+\epsilon<c<\ell\left(L_{x_o}\cap D\right)$, we can find closed, vertical segments $I_1,I_2,\dots,I_n$, contained in $L_{x_o}\cap D$, with $\sum_{k=1}^n\ell(I_k)=c$. Write $I_k=[x_o+ia_k, x_o+ib_k]$, for $k=1,2,\dots,n$, where $a_k<b_k$. Since $D$ is convex in the positive direction, the closed half-strips 
\[
\{z\in\mathbb{C}\colon \mathrm{Re}\, z \geq x_o, \; \mathrm{Im}\, z\in [a_k,b_k]\},
\]
are contained in $D$, for all $k=1,2,\dots,n$. Therefore, 
\[
\ell\left(L_x\cap D\right)>c>2y_o+\epsilon,\quad \text{for any}\quad x>x_o,
\]
which implies that $D^\sharp$ contains the half-strip 
\[
\left\{z\in\mathbb{C}\colon \mathrm{Re}\, z >x_o, \; \lvert \mathrm{Im}\, z \rvert < y_o+ \frac{\epsilon}{2}\right\}. \qedhere
\]
\end{proof}

\medskip

We proceed with the proof of Theorem \ref{T1}. From now on, we use the notations $(\phi_t), (\phi_t^\sharp), \Omega, \Omega^\sharp$ etc., set in Theorem \ref{T1}.\\

\medskip

{\it Proof of Theorem \ref{T1}}\\
By conjugating the semigroups $(\phi_t)$ and $(\phi^\sharp_t)$ with appropriate Euclidean rotations about the origin, we may and do assume  that $\tau=\tau^\sharp=1$. So, our goal is to show that for all $t>0$,
\[
\lvert \phi_t(0)-1\rvert \leq 4\pi \lvert \phi^\sharp_t(0)-1\rvert.
\]
Let $h\colon \D\to \Omega$ be the Koenigs function of $(\phi_t)$ and $h^\sharp\colon \D \to\Omega^\sharp$ the Koenigs function of $(\phi^\sharp_t)$. For $t>0$, we set
$$
A_t=\{\phi_s(0):s\ge t\}\;\;\;\hbox{and}\;\;\;A^\sharp_t=\{\phi^\sharp_s(0):s\ge t\}
$$
and note that $h(A_t)=[t,\infty)\subset \Omega$ and $h^\sharp(A^\sharp_t)=[t,\infty)\subset \Omega^\sharp$. Let $d_t$ be the diameter of $A_t$ and $\alpha_t$ be an arc of the unit circle having diameter equal to $d_t$. By Theorem \ref{diam},
\begin{eqnarray}\label{T1p1}
|\phi_t(0)-1|&\leq & d_t=2\,\sin(\pi \omega(0,\alpha_t,\D))\leq 2\pi\,\omega(0,\alpha_t,\D)\\
&\le & 2\pi\,\omega(0,A_t,\D\setminus A_t).\nonumber 
\end{eqnarray}
Also, by the conformal invariance of the harmonic measure,
\begin{equation}\label{T1p2}
\omega(0,A_t,\D\setminus A_t)=\omega(0,[t,\infty),\Omega\setminus [t,\infty)).
\end{equation}
Let $\Omega_t$ be the connected component of $\Omega\setminus \{\mathrm{Re}\ z=t\}$ that contains 0, and $\Gamma_t = \partial \Omega_t \cap \{\mathrm{Re}\ z=t\}$. The maximum principle on the domain $\Omega_t$ yields 
\[
\omega(0,[t,\infty),\Omega\setminus [t,\infty)) \leq \omega(0,\Gamma_t,\Omega_t),
\]
which in conjunction with \eqref{T1p1} and \eqref{T1p2} gives
\[
|\phi_t(0)-1|\leq 2\pi\, \omega(0,\Gamma_t,\Omega_t).
\]
By Theorem \ref{Bae}, 
\[
\omega(0,\Gamma_t,\Omega_t)\leq \omega\left(0,\Gamma_t^\sharp,\Omega_t^\sharp\right),
\]
where $\Gamma_t^\sharp$ and $\Omega_t^\sharp$ denote the Steiner symmetrizations of $\Gamma_t$ and $\Omega_t$, respectively. So, 
\begin{equation}\label{T1p3}
|\phi_t(0)-1|\leq 2\pi\,\omega\left(0,\Gamma_t^\sharp,\Omega_t^\sharp\right).
\end{equation}

Let $\mathcal{R}\Omega_t^\sharp$ denote the reflection of $\Omega_t^\sharp$ in the vertical line $\{\mathrm{Re}\ z=t\}$. Since $\Omega^\sharp$ is convex in the positive direction, we have $\mathcal{R}\Omega_t^\sharp\subset \Omega^\sharp$. We consider the domain $\widetilde{\Omega_t}:=\Omega_t^\sharp\cup \Gamma_t^\sharp\cup \mathcal{R}\Omega_t^\sharp$, which is symmetric with respect to both $\R$ and $\{\mathrm{Re}\ z=t\}$ (see Figure \ref{theorem 1 figure}). 
Therefore, the line segments $\Gamma_t^\sharp$ and $[t,\infty)\cap\widetilde{\Omega_t}$ are hyperbolic geodesic segments of $\widetilde{\Omega_t}$. Choose a Riemann map $g\colon \widetilde{\Omega_t} \to \D$, with $g(0)=0$ and $g\left([t,\infty)\cap\widetilde{\Omega_t}\right) = [r,1)$, for some $r=r(t)\in (0,1)$. Also, define $\gamma_t=g\left(\Gamma_t^\sharp\right)$, which is a hyperbolic geodesic of $\D$, perpendicular to the real axis. By conformal invariance, we have that 
\begin{equation}\label{T1p4}
\omega\left(0, \Gamma_t^\sharp, \Omega_t^\sharp\right) = \omega(0 , \gamma_t, \D\setminus \gamma_t).
\end{equation}

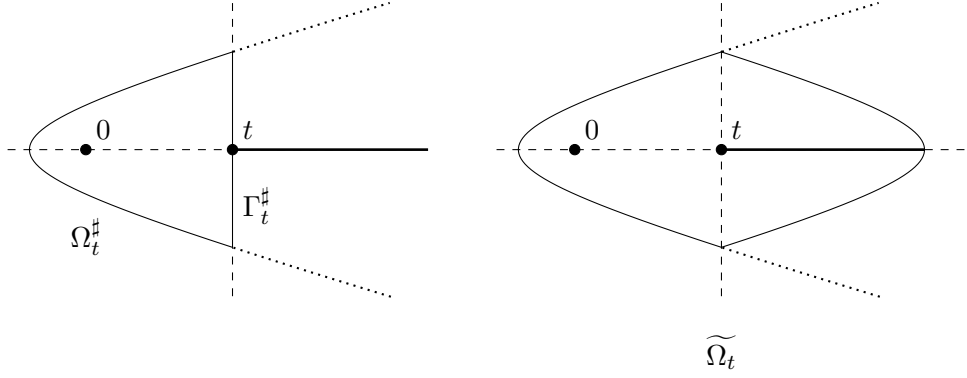
\begin{figure}[h]
\centering
\begin{tikzpicture}[scale=1.3]
\begin{scope}
\draw [samples=50,domain=-0.5925:0.5925,rotate around={0.:(0.,0.)},xshift=0.cm,yshift=0.cm] plot ({1.923347276399214*(1+(\x)^2)/(1-(\x)^2)},{0.5483933390986093*2*(\x)/(1-(\x)^2)});
\draw[dotted, line width=.8pt] [samples=50,domain=0.5925:0.7,rotate around={0.:(0.,0.)},xshift=0.cm,yshift=0.cm] plot ({1.923347276399214*(1+(\x)^2)/(1-(\x)^2)},{0.5483933390986093*2*(\x)/(1-(\x)^2)});
\draw[dotted, line width=.8pt] [samples=50,domain=-0.7:-0.5925,rotate around={0.:(0.,0.)},xshift=0.cm,yshift=0.cm] plot ({1.923347276399214*(1+(\x)^2)/(1-(\x)^2)},{0.5483933390986093*2*(\x)/(1-(\x)^2)});
\draw[dashed] (1.7,0) -- (6,0);
\draw[dashed] (4,1.5) -- (4,1);
\draw[dashed] (4,-1.5) -- (4,-1);
\draw[line width = 1pt] (4,0) -- (6,0);
\draw (4,1) -- (4,-1);
\filldraw[black] (4,0) circle (1.5pt) node[above right] {$t$};
\filldraw[black] (2.5,0) circle (1.5pt) node[above right] {0};
\node[right] at (4,-0.57) {$\Gamma^\sharp_t$};
\node[below] at (2.5,-0.6) {$\Omega^\sharp_t$};
\end{scope}

\begin{scope}[xshift=5cm]
\draw [samples=50,domain=-0.5925:0.5925,rotate around={0.:(0.,0.)},xshift=0.cm,yshift=0.cm] plot ({1.923347276399214*(1+(\x)^2)/(1-(\x)^2)},{0.5483933390986093*2*(\x)/(1-(\x)^2)});
\draw [samples=50,domain=-0.5929:0.5925,rotate around={0.:(8.,0.)},xshift=8.cm,yshift=0.cm] plot ({1.9233472763991835*(-1-(\x)^2)/(1-(\x)^2)},{0.5483933390986007*(-2)*(\x)/(1-(\x)^2)});
\draw[dotted, line width=.8pt] [samples=50,domain=0.5925:0.7,rotate around={0.:(0.,0.)},xshift=0.cm,yshift=0.cm] plot ({1.923347276399214*(1+(\x)^2)/(1-(\x)^2)},{0.5483933390986093*2*(\x)/(1-(\x)^2)});
\draw[dotted, line width=.8pt] [samples=50,domain=-0.7:-0.5925,rotate around={0.:(0.,0.)},xshift=0.cm,yshift=0.cm] plot ({1.923347276399214*(1+(\x)^2)/(1-(\x)^2)},{0.5483933390986093*2*(\x)/(1-(\x)^2)});
\draw[dashed] (1.7,0) -- (6.5,0);
\draw[dashed] (4,1.5) -- (4,-1.5);
\draw[line width = 1pt] (4,0) -- (6.08,0);
\filldraw[black] (4,0) circle (1.5pt) node[above right] {$t$};
\filldraw[black] (2.5,0) circle (1.5pt) node[above right] {0};
\node[below] at (4,-1.8) {$\widetilde{\Omega_t}$};
\end{scope}
\end{tikzpicture}
\label{theorem 1 figure}
\caption{The domains $\Omega^\sharp_t$ and $\widetilde{\Omega_t}$.}
\end{figure}

Using the M\"obius map $f(z)=-i\frac{z-r}{1-rz}$, we obtain
\begin{align}\label{T1p5}
\omega(0 , \gamma_t, \D\setminus \gamma_t) &= \omega(ir, [-1,1], \D\cap \Ha) = 1-\omega(ir,\partial\D\cap\Ha, \D\cap \Ha)\nonumber\\
	& = 1- \frac{2}{\pi}\mathrm{Arg}\left(\frac{1+ir}{1-ir}\right) = 1-\frac{2}{\pi}\tan^{-1}\left(\frac{2r}{1-r^2}\right),
\end{align}
where $\Ha$ denotes the upper half-plane. Now, from \eqref{T1p3}, \eqref{T1p4} and \eqref{T1p5}, we get
\begin{equation}\label{T1p6}
|\phi_t(0)-1| \leq 2\pi \left(1-\frac{2}{\pi}\tan^{-1}\left(\frac{2r}{1-r^2}\right)\right).
\end{equation}
But the real function
\[
\psi(r)= \frac{2}{\pi}\tan^{-1}\left(\frac{2r}{1-r^2}\right) + \frac{4}{\pi}\tan^{-1}\left(\frac{1-r}{2\sqrt{r}}\right),
\]
is strictly decreasing for $r\in(0,1)$, and $\lim\limits_{r\to1^-}\psi(r)=1$. This means that
\[
1-\frac{2}{\pi}\tan^{-1}\left(\frac{2r}{1-r^2}\right) \leq \frac{4}{\pi}\tan^{-1}\left(\frac{1-r}{2\sqrt{r}}\right), \quad \text{for all}\quad r\in(0,1),
\]
and, from \eqref{slit},
\[
\frac{2}{\pi}\tan^{-1}\left(\frac{1-r}{2\sqrt{r}}\right) = \omega(0,[r,1), \D\setminus[r,1)).
\]
Therefore, \eqref{T1p6} becomes
\begin{align}\label{T1p7}
|\phi_t(0)-1| &\leq 4\pi\, \omega(0,[r,1), \D\setminus[r,1)) \nonumber\\
	&=4\pi\, \omega\left(0,\widetilde{\Omega_t}\cap [t,\infty), \widetilde{\Omega_t}\setminus[t,\infty)\right). 
\end{align}
Using the maximum principle on the domain $\widetilde{\Omega_t}$ yields
\[
\omega\left(0,\widetilde{\Omega_t}\cap [t,\infty), \widetilde{\Omega_t}\setminus[t,\infty)\right) \leq \omega(0,[t,\infty), \Omega^\sharp\setminus [t,\infty)),
\]
meaning that 
\begin{equation}\label{T1p8}
|\phi_t(0)-1|\leq 4\pi\, \omega(0,[t,\infty), \Omega^\sharp\setminus [t,\infty)).
\end{equation}
Since $\Omega^\sharp$ is symmetric with respect to $\R$, the half-line $[t,\infty)$ lies on a hyperbolic geodesic of $\Omega^\sharp$. Hence, $h^{-1}([t,\infty))=[\phi^\sharp_t(0),1)\subset\D$, which together with \eqref{slit} implies that
\begin{align}\label{T1p9}
\omega(0,[t,\infty), \Omega^\sharp\setminus [t,\infty)) &= \omega\left(0, [\phi^\sharp_t(0),1), \D\setminus [\phi^\sharp_t(0),1) \right)\nonumber\\
	& =\frac{2}{\pi}\tan^{-1}\left(\frac{1-\phi^\sharp_t(0)}{2\sqrt{\phi^\sharp_t(0)}}\right)\nonumber\\
	&\leq 1- \phi^\sharp_t(0) = \lvert \phi^\sharp_t(0) -1 \rvert,
\end{align}
where the inequality follows from the fact that the function 
\[
k(r)=\frac{2}{\pi}\tan^{-1}\left(\frac{1-r}{2\sqrt{r}}\right)-1+r,\quad r\in(0,1)
\]
has a unique minimum in $(0,1)$ and $\lim\limits_{r\to0^+}k(r)=\lim\limits_{r\to1^-}k(r)=0$. Combining \eqref{T1p8} with \eqref{T1p9} yields the desired conclusion.


\section{{\bf Proof of Theorem \ref{T2}}}

The following is a simple observation. 

\begin{lm}\label{T2L1}
If $A, B$ are sets that are convex in the positive direction, then the same is true for the sets $A^*$ and $A\cup B$. If, in addition, $A\cap B\neq \emptyset$, then $A\cap B$ is convex in the positive direction.
\end{lm}

\begin{lm}\label{T2L2}
If $D$ is a domain that is convex in the positive direction, then its polarization $\widehat{D}$ is also a domain, convex in the positive direction.
\end{lm}

\begin{proof}
The fact that $\widehat{D}$ is convex in the positive direction follows from Lemma~\ref{T2L1}. Also, it is known that $\widehat{D}$ is open, whenever $D$ is open. So, we only need to prove that $\widehat{D}$ is connected. If $D\cap\R=\emptyset$, then either $\widehat{D}=D$, or $\widehat{D}=D^*$, and the result follows. Suppose, now, that $D\cap\R\neq\emptyset$. Then, since $D$ and $D^*$ are convex in the positive direction, there exists some $x_0\in \R$, so that the half-line $\{x\in \R\colon x \geq x_0\}$ is contained in $D$ and $D^*$. This implies that $\left(D\cup D^*\right)^+$ is connected. To complete the proof, we need to show that $\left(D\cap D^*\right)^-$ is connected. Let $z_1,z_2\in \left(D\cap D^*\right)^-$. Since $D$ is connected, there exists a path $\gamma_1$ in $D$ that joins $z_1$ and $x_0$. Write $x_1=\max\{\mathrm{Re}\, z \colon z\in \gamma_1\}$, and consider the half-strip
\[
S_1=\{x+iy\in\mathbb{C}\colon\ x\geq x_1,\ \mathrm{Im}\, z_1\leq y\leq0\}.
\]
Because $D$ is convex in the positive direction, $S_1$ is contained in $D$. The same arguments can be applied to $D^*$, to show that we can find $x_1^*\in\R$, so that the half-strip
\[
S_1^*=\{x+iy\in\mathbb{C}\colon\ x\geq x_1^*,\ \mathrm{Im}\, z_1\leq y\leq0\},
\]
is contained in $D^*$. Therefore the half-strip $S_1\cap S_1^*$ is contained in $\left(D\cap D^*\right)^-$. Working similarly for $z_2$, we can show that there exists $x_2\in\R$, so that the half-strip
\[
\{x+iy\in\mathbb{C}\colon\ x\geq x_2,\ \mathrm{Im}\, z_2\leq y\leq0\},
\] 
is contained in $\left(D\cap D^*\right)^-$. The set $\left(D\cap D^*\right)^-$ is also convex in the positive direction, due to Lemma~\ref{T2L1}, and so the half-lines $\{z_1+t\colon t\geq0\}$ and $\{z_2+t\colon t\geq0\}$, both lie in $\left(D\cap D^*\right)^-$. Hence, we can join $z_1$ and $z_2$ with a polygonal path in $\left(D\cap D^*\right)^-$, consisting of two horizontal line segments and a vertical line segment, implying that $\left(D\cap D^*\right)^-$ is path-connected.
\end{proof}

{\it Proof of Theorem \ref{T2}}\\
As before, we conjugate $(\phi_t)$ and $(\widehat{\phi}_t)$, so that $\tau=\widehat{\tau}=1$. Let $h\colon \D \to \Omega$ be the Koenigs function of $(\phi_t)$ and $\widehat{h}\colon \D \to \widehat{\Omega}$ the Koenigs function of $(\widehat{\phi}_t)$. For $t>0$, we set
\[
A_t=\{\phi_s(0):s\ge t\}\;\;\;\hbox{and}\;\;\;\widehat{A}_t=\{\widehat{\phi}_s(0):s\ge t\},
\]
and note that $h(A_t)=[t,\infty)\subset \Omega$ and $\widehat{h}(\widehat{A}_t)=[t,\infty)\subset \widehat{\Omega}$. Then, following the same arguments as the proof of Theorem \ref{T1}, we have that
\begin{equation}\label{T2p1}
\lvert \phi_t(0) - 1 \rvert \leq 2\pi \, \omega(0,A_t, \D\setminus A_t) = 2\pi \, \omega(0,[t,\infty), \Omega\setminus [t,\infty)).
\end{equation}
Due to Theorem \ref{pol} and conformal invariance, 
\[
\omega(0,[t,\infty), \Omega\setminus [t,\infty)) \leq \omega(0,[t,\infty), \widehat{\Omega}\setminus [t,\infty)) = \omega(0, \widehat{A}_t, \D\setminus \widehat{A}_t).
\]
So, \eqref{T2p1} yields
\begin{equation}\label{T2p2}
\lvert \phi_t(0) - 1 \rvert \leq 2\pi\, \omega(0, \widehat{A}_t, \D\setminus \widehat{A}_t).
\end{equation}
Let $\widehat{\Gamma}_t$ be the hyperbolic geodesic of $\D$, passing through $\widehat{\phi}_t(0)$, that is perpendicular to the diameter $(-1,1)$. Denote by $\D_t$ the connected component of $\D\setminus\widehat{\Gamma}_t$ that contains 0. 
It is a consequence of \cite[Theorem 1.1]{BK} that $\widehat{A}_t\subset \D\setminus \D_t$. Therefore, 
by the maximum principle on the domain $\D_t$, we obtain 
\[
\omega(0,\widehat{A}_t, \D\setminus \widehat{A}_t) \leq \omega(0, \widehat{\Gamma}_t, \D_t).
\]
If $r=r(t)\in(-1,1)$ is the intersection of $\widehat{\Gamma}_t$ with $(-1,1)$, then 
\[
\omega(0, \widehat{\Gamma}_t, \D_t) = 1-\frac{2}{\pi}\tan^{-1}\left(\frac{2r}{1-r^2}\right).
\]
But $1-\frac{2}{\pi}\tan^{-1}(x) = \frac{2}{\pi}\tan^{-1}(1/x)$, for all $x\in\R\setminus \{0\}$. So,
\[
\omega(0, \widehat{\Gamma}_t, \D_t) = \frac{2}{\pi}\tan^{-1}\left(\frac{1-r^2}{2r}\right).
\]

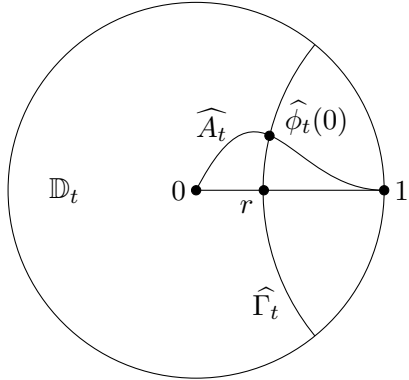
\begin{figure}[h]
\begin{tikzpicture}[scale=2.5]
\draw  (0.,0.) circle (1.cm);
\draw  (0.,0.)-- (1.,0.);
\draw [shift={(1.5834491914722626,0.)}]  plot[domain=2.458064196749281:3.825121110430305,variable=\t]({1.*1.2277260858897487*cos(\t r)+0.*1.2277260858897487*sin(\t r)},{0.*1.2277260858897487*cos(\t r)+1.*1.2277260858897487*sin(\t r)});
\draw (0,0) .. controls (0.36,0.7) and (0.5,0) .. (1,0);
\node[left] at (0.22,0.35) {$\widehat{A_t}$};
\node[left] at (0.5,-0.6) {$\widehat{\Gamma_t}$};
\node at (-.7,0) {$\D_t$};
\filldraw[black] (0.36,0) circle (.7pt) node[below left] {$r$};
\filldraw[black] (0.39,0.29) circle (.7pt) node[above right, xshift=2pt, yshift=-4pt] {$\widehat{\phi_t}(0)$};
\filldraw[black] (0,0) circle (.7pt) node[left] {$0$};
\filldraw[black] (1,0) circle (.7pt) node[right] {$1$};
\end{tikzpicture}
\caption{The hyperbolic geodesic $\widehat{\Gamma_t}$ and the domain $\D_t$.}
\end{figure}

Therefore, \eqref{T2p2} becomes
\[
\lvert \phi_t(0) - 1 \rvert \leq 4  \tan^{-1}\left(\frac{1-r^2}{2r}\right).
\]
An elementary calculus argument shows that
\[
\tan^{-1}\left(\frac{1-x^2}{2x}\right) \leq \frac{\pi}{2} (1-x), \quad \text{for all} \quad x\in(-1,0)\cup(0,1),
\]
which implies 
\begin{equation}\label{T2p3}
\lvert \phi_t(0) - 1 \rvert  \leq 2\pi\, (1-r).
\end{equation}
Suppose that $\widehat{C}_t$ is the Euclidean circle in $\C$ that contains the arc $\widehat{\Gamma}_t$. It is a basic fact of hyperbolic geometry that the center of $\widehat{C}_t$ is a real number that lies in the exterior of $\D$. Therefore,
\[
(1-r)=\mathrm{dist}\left(\widehat{C}_t, 1\right)=\mathrm{dist}\left(\widehat{\Gamma}_t, 1\right)\leq |\widehat{\phi}_t(0)-1|,
\]
which, together with \eqref{T2p3}, completes the proof.


\end{document}